\documentclass[10pt]{article}
\usepackage{amsthm,amstext,amssymb,amsmath,graphicx,amsfonts,caption,subcaption,mathrsfs,multirow,dsfont}
\usepackage{tikz}
\usetikzlibrary{matrix}
\usepackage{multicol}
\usepackage{xcolor}
\usepackage[sort&compress,round,comma]{natbib}
\usepackage[pagebackref=false,colorlinks,linkcolor=blue,citecolor=blue]{hyperref}
\theoremstyle{plain}
\newtheorem{theorem}{Theorem}[section]
\newtheorem{lemma}[theorem]{Lemma}
\newtheorem{proposition}[theorem]{Proposition}

\theoremstyle{definition}
\newtheorem{definition}[theorem]{Definition}
\newtheorem{corollary}[theorem]{Corollary}
\newtheorem{example}[theorem]{Example}

\theoremstyle{remark}
\newtheorem{remark}{\sc Remark}
\makeatother
\makeatletter
\def\namedlabel#1#2{\begingroup
   \def\@currentlabel{#2}%
   \label{#1}\endgroup}
\makeatother
\date{}
\title{\bf  Quasicomplemented residuated lattices}\vspace{.25 in}
\author{ \vspace{.25 in} {\bf  Saeed Rasouli}\\
Department of Mathematics,
Persian Gulf University, \\Bushehr, 75169, IRAN\\
{\tt srasouli@pgu.ac.ir }\\\
}
 \begin{document}
 \maketitle
 \begin{abstract}
  In this paper, the class of quasicomplemented residuated lattices is introduced and investigated, as a subclass of residuated lattices in which any prime filter not containing any dense element is a minimal prime filter. The notion of disjunctive residuated lattices is introduced and it is observed that a residuated lattice is Boolean if and only if it is disjunctive and quasicomplemented. Finally, some characterizations for quasicomplemented residuated lattices are given by means of the new notion of $\alpha$-filters.
\end{abstract}
\section{Introduction}

As a generalization of distributive pseudo-complemented lattices, \cite{var1} studied lattices which are just pseudo-complemented. He observed that a distributive lattice is pseudo-complemented if and only if each of its annulet is principal. By this motivation, \cite{val} introduced the notion of quasi-complemented lattices as a generalization of distributive pseudo-complemented lattices. Also, \cite{spe0} introduced the class of $\star$-lattices as a subclass of distributive lattices. \citet[Proposition 3.4]{spe} proved that these two classes are equivalent. Quasi-complemented lattices are studied extensively by \cite{cor,cor1,jay,spe2}. Also, this notion is discussed for rings in \cite{kno}. In this paper, we introduce the notion of quasicomplemented residuated lattice and generalize some results of \cite{cor1} and \cite{spe} in this class of algebras.

This paper is organized in five sections as follow: In Sec. \ref{sec2}, some definitions and facts about residuated lattices are recalled and some of their propositions are proved. In Sec. \ref{sec3}, the notion of quasicomplemented residuated lattices, as a subclass of residuated lattices, is introduced and some of their properties are investigated. In Sec. \ref{sec4}, notions of disjunctive and weakly disjunctive residuated lattices are introduced and some of their characterizations are derived. It is proved that the lattice of a residuated lattice principal filters is Boolean if and only if it is weakly disjunctive and quasicomplemented. In Sec. \ref{sec5}, the notion of $\alpha$-filters is introduced and some of their properties are studied. Weakly disjunctive residuated lattices are characterized in terms of $\alpha$-filters and it is shown that a residuated lattice is quasicomplemented if and only if any its prime $\alpha$-filter is a minimal prime filter. We end this paper by deriving a set of equivalent conditions for any $\alpha$-filter to be principal.
\section{Definitions and first properties}\label{sec2}

In this section, we recall some definitions, properties and results relative to residuated lattices, which will be used
in the following.

An algebra $\mathfrak{A}=(A;\vee,\wedge,\odot,\rightarrow,0,1)$ is called a \textit{residuated lattice} if $\ell(\mathfrak{A})=(A;\vee,\wedge,0,1)$ is a bounded lattice, $(A;\odot,1)$ is a commutative monoid and $(\odot,\rightarrow)$ is an adjoint pair. In a residuated lattice $\mathfrak{A}$, for any $a\in A$, we put $\neg a:=a\rightarrow 0$ and for any integer $n$ we right $x^{n}$ instead of $x\odot\cdots\odot x$ ($n$ times). An element $\mathfrak{A}$ is called \textit{nilpotent} if $x^{n}=0$ for an integer $n$. The set of nilpotent elements of $\mathfrak{A}$ shall be denoted by $N(\mathfrak{A})$. It is well-known that $N(\mathfrak{A})$ is an ideal of $\ell(\mathfrak{A})$. \cite{idz} showed that the class of residuated lattices is equational, and so it forms a variety. The properties of residuated lattices were presented in \cite{gal}. For a survey of residuated lattices we refer to \cite{jip}.
\begin{remark}\label{resproposition}\citep[Proposition 2.2]{jip}
Let $\mathfrak{A}$ be a residuated lattice. The following conditions are satisfied for any $x,y,z\in A$:
\begin{enumerate}
  \item [$r_{1}$ \namedlabel{res1}{$r_{1}$}] $x\odot (y\vee z)=(x\odot y)\vee (x\odot z)$;
  \item [$r_{2}$ \namedlabel{res2}{$r_{2}$}] $x\vee (y\odot z)\geq (x\vee y)\odot (x\vee z)$.
  \end{enumerate}
\end{remark}

Let $\mathfrak{A}$ be a residuated lattice. The set of all complemented elements in $\ell(\mathfrak{A})$ is denoted by $B(\mathfrak{A})$ and it is called the Boolean center of $\mathfrak{A}$. For a survey of residuated lattices we refer to \cite{ciu1}.
\begin{proposition}\cite{ciu1}\label{1propos}
Let $\mathfrak{A}$ be a residuated lattice. The following assertions hold for any $e\in B(\mathfrak{A})$ and $a\in A$:
\begin{enumerate}
  \item\label{1propos1} $e^c=\neg e$;
  \item\label{1propos2} $e^n=e$, for each integer $n$;
  \item\label{1propos3} $e\odot a=a\odot e=e\wedge a$;
  \item\label{1propos4} $\neg\neg e=e$.
\end{enumerate}
\end{proposition}
\begin{example}\label{rex1}
Let $A_7=\{0,a,b,c,d,e,1\}$ be a lattice whose Hasse diagram is below (see Figure \ref{graph7}).  Define $\odot$ and $\rightarrow$ on $A_7$ as follows:
\begin{eqnarray*}
\begin{array}{l|lllllll}
  \odot & 0 & a & b & c & d & e & 1   \\ \hline
    0   & 0 & 0 & 0 & 0 & 0 & 0 & 0  \\
    a   & 0 & a & a & a & a & a & a  \\
    b   & 0 & a & b & a & b & a & b  \\
    c   & 0 & a & a & a & a & c & c  \\
    d   & 0 & a & b & a & b & c & d  \\
    e   & 0 & a & a & c & c & e & e  \\
    1   & 0 & a & b & c & d & e & 1
  \end{array}& \hspace{1cm} &
  \begin{array}{l|lllllll}
  \rightarrow & 0 & a & b & c & d & e & 1   \\ \hline
    0         & 1 & 1 & 1 & 1 & 1 & 1 & 1  \\
    a         & 0 & 1 & 1 & 1 & 1 & 1 & 1  \\
    b         & 0 & e & 1 & e & 1 & e & 1  \\
    c         & 0 & d & d & 1 & 1 & 1 & 1  \\
    d         & 0 & c & d & e & 1 & e & 1  \\
    e         & 0 & b & b & d & d & 1 & 1  \\
    1         & 0 & a & b & c & d & e & 1
  \end{array}
\end{eqnarray*}
\begin{figure}[h]
\centering
\includegraphics[scale=.15]{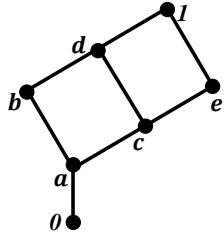}
\caption{The Hasse diagram of $\mathfrak{A}_7$.}
\label{graph7}
\end{figure}
Routine calculation shows that  $\mathfrak{A}_7=(A_7;\vee,\wedge,\odot,\rightarrow,0,1)$ is a residuated lattice.
\end{example}

Let $\mathfrak{A}$ be a residuated lattice. A non-void subset $F$ of $A$ is called a \textit{filter} of $\mathfrak{A}$ if $x,y\in F$ implies $x\odot y\in F$ and $x\vee y\in F$ for any $x\in F$ and $y\in A$. The set of filters of $\mathfrak{A}$ is denoted by $\mathscr{F}(\mathfrak{A})$. A filter $F$ of $\mathfrak{A}$ is called \textit{proper} if $F\neq A$. Clearly, $F$ is a proper filter if and only if $0\notin F$. For any subset $X$ of $A$ the \textit{filter of $\mathfrak{A}$ generated by $X$} is denoted by $\mathscr{F}(X)$. For each $x\in A$, the filter generated by $\{x\}$ is denoted by $\mathscr{F}(x)$ and called \textit{principal filter}. The set of principal filters is denoted by $\mathscr{PF}(\mathfrak{A})$. Let $\mathcal{F}$ be a collection of filters of $\mathfrak{A}$. Set $\veebar \mathcal{F}=\mathscr{F}(\cup \mathcal{F})$. It is well-known that $(\mathscr{F}(\mathfrak{A});\cap,\veebar,\textbf{1},A)$ is a frame and so it is a complete Heyting algebra.
\begin{example}\label{fex1}
Consider the residuated lattice $\mathfrak{A}_7$ from Example \ref{rex1}. Then $\mathscr{F}(\mathfrak{A}_7)=\{F_1=\{1\},F_2=\{b,d,1\},F_3=\{e,1\},F_4=\{a,b,c,d,e,1\},F_5=A_7\}$.
\end{example}

The following remark has a routine verification.
\begin{remark}\label{genfilprop}
Let $\mathfrak{A}$ be a residuated lattice and $F$ be a filter of $\mathfrak{A}$. The following assertions hold for any $x,y\in A$:
\begin{enumerate}
  \item  [$(1)$ \namedlabel{genfilprop1}{$(1)$}] $\mathscr{F}(F,x):=F\veebar \mathscr{F}(x)=\{a\in A|f\odot x^n\leq a,~f\in F\}$;
  \item  [$(2)$ \namedlabel{genfilprop2}{$(2)$}] $x\leq y$ implies $\mathscr{F}(F,y)\subseteq \mathscr{F}(F,x)$.
  \item  [$(3)$ \namedlabel{genfilprop3}{$(3)$}] $\mathscr{F}(F,x)\cap \mathscr{F}(F,y)=\mathscr{F}(F,x\vee y)$;
  \item  [$(4)$ \namedlabel{genfilprop4}{$(4)$}] $\mathscr{F}(F,x)\veebar \mathscr{F}(F,y)=\mathscr{F}(F,x\odot y)$;
  \item  [$(5)$ \namedlabel{genfilprop5}{$(5)$}] $\mathscr{PF}(\mathfrak{A})$ is a sublattice of $\mathscr{F}(\mathfrak{A})$;
  \item  [$(6)$ \namedlabel{genfilprop6}{$(6)$}] $\mathscr{F}(e)=\{a\in A|e\leq a\}$, for any $e\in B(\mathfrak{A})$.
\end{enumerate}
\end{remark}

A proper filter of a residuated lattice $\mathfrak{A}$ is called \textit{maximal} if it is a maximal element in the set of all proper filters. The set of all maximal filters of $\mathfrak{A}$ is denoted by $Max(\mathfrak{A})$. A proper filter $P$ of $\mathfrak{A}$ is called \textit{prime}, if for any $x,y\in A$, $x\vee y\in P$ implies $x\in P$ or $y\in P$. The set of all prime filters of $\mathfrak{A}$ is denoted by $Spec(\mathfrak{A})$. Since $\mathscr{F}(\mathfrak{A})$ is a distributive lattice, so $Max(\mathfrak{A})\subseteq Spec(\mathfrak{A})$. By Zorn's lemma follows that any proper filter is contained in a maximal filter and so in a prime filter. A non-empty subset $\mathscr{C}$ of $\mathfrak{A}$ is called \textit{$\vee$-closed} if it is closed under the join operation, i.e $x,y\in \mathscr{C}$ implies $x\vee y\in \mathscr{C}$.
\begin{theorem}\cite[Theorem 2.4]{raskon}\label{prfilth}
If $\mathscr{C}$ is a $\vee$-closed subset of $\mathfrak{A}$ which does not meet the filter $F$, then $F$ is contained in a filter $P$ which is maximal with respect to the property of not meeting $\mathscr{C}$; furthermore $P$ is prime.
\end{theorem}

Let $\mathfrak{A}$ be a residuated lattice and $X$ be a subset of $A$. A prime filter $P$ is called a \textit{minimal prime filter belonging to $X$} or $X$-\textit{minimal prime filter} if $P$ is a minimal element in the set of prime filters containing $X$. The set of $X$-minimal prime filters of $\mathfrak{A}$ is denoted by $Min_{X}(\mathfrak{A})$. A prime filter $P$ is called a \textit{minimal prime} if $P\in Min_{\{1\}}(\mathfrak{A})$. The set of minimal prime filters of $\mathfrak{A}$ is denoted by $Min(\mathfrak{A})$. For the basic facts concerning minimal prime filters of a residuated lattice belonging to a filter we refer to \cite{raskon}.
\begin{theorem}\cite[Theorem 2.6]{raskon}\label{1mineq}
Let $\mathfrak{A}$ be a residuated lattice. A subset $P$ of $A$ is a minimal prime filter if and only if $P^{c}$ is a $\vee$-closed subset of $\mathfrak{A}$ which it is maximal with respect to the property of not containing $1$.
\end{theorem}

Let $\mathfrak{A}$ be a residuated lattice. For a subset $\mathfrak{M}$ of $Min(\mathfrak{A})$ we write $k(\mathfrak{M})=\bigcap \{\mathfrak{m}|\mathfrak{m}\in \mathfrak{M}\}$ and for a subset $X$ of $A$ we write $h(X)=\{\mathfrak{m}|X\subseteq \mathfrak{m}\}$. Also, let $d(X)=Min(\mathfrak{A})\setminus h(X)$. If the collection $\{h(x)|x\in A\}$ is taken as a closed basis, the resulting topology is called \textit{the hull-kernel topology} which is denoted by $\tau_{h}$, and if the collection $\{h(x)|x\in A\}$ is taken as an open basis, the resulting topology is called \textit{the dual hull-kernel topology} which is denoted by $\tau_{d}$. For a detailed discussion of spaces of minimal prime filters in residuated lattices we refer to \cite{rasdeh}.
\begin{corollary}\cite[Corollary 4.4]{rasdeh}\label{minspapro}
Let $\mathfrak{A}$ be a residuated lattice. The following assertions hold:
 \begin{enumerate}
   \item [(1) \namedlabel{minspapro1}{(1)}] $(Min(\mathfrak{A});\tau_{h})$ is zero-dimensional and consequently totally disconnected;
   \item [(2) \namedlabel{minspapro2}{(2)}] $\tau_{d}$ is finer than $\tau_{h}$.
 \end{enumerate}
 \end{corollary}

Let $\mathfrak{A}$ be a residuated lattice. For any subset $X$ of $A$ we write $X^{\perp}=\{a\in A|a\vee x=1,\forall x\in X\}$. We set $\Gamma(\mathfrak{A})=\{X^{\perp}|X\subseteq A\}$ and $\gamma(\mathfrak{A})=\{x^{\perp}|x\in A\}$. Elements of $\Gamma(\mathfrak{A})$ and $\gamma(\mathfrak{A})$ are called \textit{coannihilators} and \textit{coannulets}, respectively. By \citet[Proposition 3.13]{rascoan} follows that $(\Gamma(\mathfrak{A});\cap,\vee^{\Gamma},\{1\},A)$ is a complete Boolean lattice, where for any $\mathscr{G}\subseteq \Gamma(\mathfrak{A})$ we have $\vee^{\Gamma} \mathscr{G}=(F:(F:\cup \mathscr{G}))$, and by \citet[Corollary 3.14]{rascoan} follows that $\gamma(\mathfrak{A})$ is a sublattice of $\Gamma(\mathfrak{A})$. A subset $X$ of $A$ is called \textit{dense} if $X^{\perp}=\{1\}$. The set of all dense elements of $\mathfrak{A}$ shall be denoted by $\mathfrak{d}(\mathfrak{A})$. It is well-known that $\mathfrak{d}(\mathfrak{A})$ is an ideal of $\ell(\mathfrak{A})$.
\begin{proposition}\citep{rascoan}\label{1fxpro}
Let $\mathfrak{A}$ be a residuated lattice. The following assertions hold for any $X,Y\subseteq A$ and $F,G\in \mathscr{F}(\mathfrak{A})$:
\begin{enumerate}
\item  [$(1)$ \namedlabel{1fxpro1}{$(1)$}] $X\subseteq Y^{\perp}$ implies $Y\subseteq X^{\perp}$;
\item  [$(2)$ \namedlabel{1fxpro2}{$(2)$}] $X^{\perp}\cap X^{\perp\perp}=\{1\}$;
\item  [$(3)$ \namedlabel{1fxpro3}{$(3)$}] $X\subseteq X^{\perp\perp}$;
\item  [$(4)$ \namedlabel{1fxpro4}{$(4)$}] $(\mathscr{F}(X))^{\perp}=X^{\perp}$;
\end{enumerate}
\end{proposition}
\begin{proposition}\label{9fxpro}
Let $\mathfrak{A}$ be a residuated lattice and $F$ be a filter of $\mathfrak{A}$. Then $F^{\perp}$ is the pseudocomplement of $F$.
\end{proposition}
\begin{proof}
By Proposition \ref{1fxpro}, it follows that $F\cap F^{\perp}=\{1\}$. Assume that $F\cap G=\{1\}$ for some filter $G$. Let $a\in G$ and $b\in F$. Since $a,b\leq a\vee b$ so $a\vee b\in F\cap G=\{1\}$. Thus $a\in F^{\perp}$ and it shows that $G\subseteq F^{\perp}$.
\end{proof}
\begin{corollary}\label{comincar}\citep[Corollary 2.11]{rasdeh}
Let $\mathfrak{A}$ be a residuated lattice. Then, for any subset $X$ of $A$, we have
\[X^{\perp}=\bigcap\{\mathfrak{m}\in Min(\mathfrak{A})|~X\nsubseteq \mathfrak{m}\}.\]
\end{corollary}
\begin{proposition}\citep[Proposition 3.15]{rascoan}\label{4fxpro}
Let $\mathfrak{A}$ be a residuated lattice. The following assertions hold for any $x,y\in A$:
\begin{enumerate}
\item  [$(1)$ \namedlabel{4fxpro1}{$(1)$}] $x\leq y$ implies $x^{\perp}\subseteq y^{\perp}$;
\item  [$(2)$ \namedlabel{4fxpro2}{$(2)$}] $x^{\perp}\cap y^{\perp}=(x\odot y)^{\perp}$;
\item  [$(3)$ \namedlabel{4fxpro3}{$(3)$}] $x^{\perp\perp}\cap y^{\perp\perp}=(x\vee y)^{\perp\perp}$;
\item  [$(4)$ \namedlabel{4fxpro4}{$(4)$}] $x^{\perp}\veebar y^{\perp}\subseteq x^{\perp}\vee^{\Gamma} y^{\perp}=(x\vee y)^{\perp}$;
\item  [$(5)$ \namedlabel{4fxpro5}{$(5)$}] $e^{\perp}=\mathscr{F}(\neg e)$, for any $e\in B(\mathfrak{A})$.
\end{enumerate}
\end{proposition}
\begin{proposition}\label{nondepri}
Let $\mathfrak{A}$ be a residuated lattice. Then any non-dense prime filter of $\mathfrak{A}$ is a coannulet.
\end{proposition}
\begin{proof}
Let $P$ be a non-dense prime filter of $\mathfrak{A}$. So $P^{\perp}\neq \textbf{1}$. So there exists $1\neq x\in P^{\perp}$. Thus $P\subseteq P^{\perp\perp}\subseteq x^{\perp}$. Otherwise, $y\in x^{\perp}$ implies that $x\vee y=1\in P$. But $x\notin P$ since $x\in P$ states that $x\in P\cap P^{\perp}=\{1\}$ and it means that $x=1$ which it is a contradiction. So $y\in P$ and it shows that $P=x^{\perp}$.
\end{proof}

Let $\mathfrak{A}$ be a residuated lattice. For any ideal $I$ of $\ell(\mathfrak{A})$ we write $\omega(I)=\{a\in A|a\vee x=1,\exists x\in I\}$. We set $\Omega(\mathfrak{A})=\{\omega(I)|I\textrm{~is~an~ideal~of~}\ell(\mathfrak{A})\}$. By \citet[Proposition 3.3]{raskon} follows that $(\Omega(\mathfrak{A});\cap,\vee^{\omega},\{1\},A)$ is a bounded distributive lattice, where $F\vee^{\omega} G=\omega(I_{F}\curlyvee I_{G})$, for any $F,G\in \Omega(\mathfrak{A})$ (by $\curlyvee$, we mean the join operation in the lattice of ideals of $\ell(\mathfrak{A})$). Also, by \citet[Proposition 3.8]{raskon} follows that $\gamma(\mathfrak{A})$ is a sublattice of $\Omega(\mathfrak{A})$. For a prime filter $P$ of $\mathfrak{A}$, we write $D(P)=\omega(P^{c})$. The following corollary is a characterization for minimal prime filters.
\begin{proposition}\label{profilnoden}
A proper $\omega$-filter in a residuated lattice $\mathfrak{A}$ contains no dense elements.
\end{proposition}
\begin{proof}
Let $F$ be a proper $\omega$-filter. Assume that $F$ contains a dense element as $d$. Hence $d\in x^{\perp}$ for some $x\in I_F$. It implies that $x\in d^{\perp}$ and it means $1\in I_{F}$. So $F=A$; a contradiction.
\end{proof}
\begin{theorem}\cite[Proposition 3.14]{raskon}\label{mincor}
Let $\mathfrak{A}$ be a residuated lattice. The following assertions are equivalent:
\begin{enumerate}
  \item  [$(1)$ \namedlabel{mincor1}{$(1)$}] $P$ is a minimal prime filter;
  \item  [$(2)$ \namedlabel{mincor2}{$(2)$}] $P=D(P)$;
  \item  [$(3)$ \namedlabel{mincor3}{$(3)$}] for any $x\in A$, $P$ contains precisely one of $x$ or $x^{\perp}$.
\end{enumerate}
\end{theorem}
\begin{corollary}\label{mindperp}
Let $\mathfrak{A}$ be a residuated lattice. The following assertions are equivalent for any $x,y\in A$:
\begin{enumerate}
  \item  [$(1)$ \namedlabel{mindperp1}{$(1)$}] $x^{\perp}=y^{\perp}$;
  \item  [$(2)$ \namedlabel{mindperp2}{$(2)$}] $h(x)=h(y)$;
  \item  [$(3)$ \namedlabel{mindperp3}{$(3)$}] $d(x)=d(y)$.
\end{enumerate}
\end{corollary}
\begin{proof}
\item [\ref{mindperp1}$\Rightarrow$\ref{mindperp2}:] Let $\mathfrak{m}\in h(x)$. By Theorem \ref{mincor}\ref{mincor3} follows that $y^{\perp}\nsubseteq \mathfrak{m}$ and so $y\in \mathfrak{m}$. It implies that $\mathfrak{m}\in h(y)$. Thus we have $h(x)\subseteq h(y)$. The other inclusion is analogous by symmetry.
\item [\ref{mindperp2}$\Rightarrow$\ref{mindperp3}:] It is evident.
\item [\ref{mindperp3}$\Rightarrow$\ref{mindperp1}:] by Corollary \ref{comincar} follows that $x^{\perp}=\bigcap d(x)=\bigcap d(y)=y^{\perp}$.
\end{proof}

\begin{corollary}\cite[Proposition 3.21]{raskon}\label{dfmpr}
Let $\mathfrak{A}$ be a residuated lattice. We have $D(P)=\bigcap \{\mathfrak{m}\in Min(\mathfrak{A})|\mathfrak{m}\subseteq P\}$.
\end{corollary}
\section{Quasicomplemented residuated lattices}\label{sec3}

In this section we introduce and study the notion of quasicomplemented residuated lattices.
\begin{definition}
Let $\mathfrak{A}$ be a residuated lattice. $\mathfrak{A}$ is called \textit{quasicomplemented} provided that for any $x\in A$, there exists $y\in A$ such that $x^{\perp\perp}=y^{\perp}$.
\end{definition}
\begin{proposition}\label{qucomeq}
Let $\mathfrak{A}$ be a residuated lattice. The following assertions are equivalent:
\begin{enumerate}
\item  [$(1)$ \namedlabel{qucomeq1}{$(1)$}] $\mathfrak{A}$ is quasicomplemented;
\item  [$(2)$ \namedlabel{qucomeq2}{$(2)$}] for any $x\in A$, there exists $y\in A$ such that $x\odot y\in \mathfrak{d(A)}$ and $x\vee y=1$;
\item  [$(3)$ \namedlabel{qucomeq3}{$(3)$}] $\gamma(\mathfrak{A})$ is a Boolean lattice.
\end{enumerate}
\end{proposition}

\begin{proof}
\begin{enumerate}
\item  []\ref{qucomeq1}$\Rightarrow$\ref{qucomeq2}: Consider $x\in A$. So there exists $y\in A$ such that $x^{\perp\perp}=y^{\perp}$. Applying Proposition \ref{1fxpro}\ref{1fxpro2} and \ref{4fxpro}\ref{4fxpro2}, it follows that $x\odot y$ is a dense element and Proposition \ref{1fxpro}\ref{1fxpro2} shows that $x\in x^{\perp\perp}=y^{\perp}$ and so $x\vee y=1$.
\item  []\ref{qucomeq2}$\Rightarrow$\ref{qucomeq3}: By applying Proposition \ref{4fxpro}(\ref{4fxpro2} and \ref{4fxpro4}), it is evident.
\item  []\ref{qucomeq3}$\Rightarrow$\ref{qucomeq1}: Let $x\in A$. So there exists some $y\in A$ such that $x^{\perp}\cap y^{\perp}=\{1\}$ and $x^{\perp}\vee^{\Gamma} y^{\perp}=A$. Applying Proposition \ref{9fxpro}, the former states $y^{\perp}\subseteq x^{\perp\perp}$, and the latter states the reverse inclusion.
\end{enumerate}
\end{proof}

In the following, we derive a sufficient condition for a residuated lattice to become quasicomplemented.
\begin{proposition}\label{coannprinstar}
Let $\mathfrak{A}$ be a residuated lattice. $\mathfrak{A}$ is quasicomplemented provided that in which any coannulet is principal.
\end{proposition}
\begin{proof}
Consider $x\in A$. So there exist $a\in A$ such that $x^{\perp}=\mathscr{F}(y)$. Using Proposition \ref{1fxpro}\ref{1fxpro4}, it follows that $x^{\perp\perp}=y^{\perp}$ and so $\mathfrak{A}$ is quasicomplemented.
\end{proof}

In the following proposition, we derive a necessary and sufficient condition for any residuated lattice to become quasicomplemented.
\begin{proposition}\label{44fxpro}
Let $\mathfrak{A}$ be a residuated lattice and $F$ be a filter of $\mathfrak{A}$. The following assertions are equivalent:
\begin{enumerate}
\item  [$(1)$ \namedlabel{44fxpro1}{$(1)$}] $\mathfrak{A}$ is quasicomplemented;
\item  [$(2)$ \namedlabel{44fxpro2}{$(2)$}] any prime filter not containing any dense element is minimal prime;
\item  [$(3)$ \namedlabel{44fxpro3}{$(3)$}] any filter not containing any dense element is contained in a minimal prime filter.
\end{enumerate}
\end{proposition}
\begin{proof}
\begin{enumerate}
\item  []\ref{44fxpro1}$\Rightarrow$\ref{44fxpro2}: Let $P$ be a prime filter such that $P\cap \mathfrak{d}(\mathfrak{A})=\emptyset$. Consider $x\in P$. Applying Proposition \ref{qucomeq}\ref{qucomeq2}, there exists $y\in A$ such that $x\odot y$ is dense and $x\vee y=1$. It shows that $x\odot y\notin P$ and so $y\notin P$. Hence, $x\in D(P)$ and it states that $P=D(P)$. So the result holds by Theorem \ref{mincor}.
\item  []\ref{44fxpro2}$\Rightarrow$\ref{44fxpro3}: It follows by Theorem \ref{prfilth}.
\item  []\ref{44fxpro3}$\Rightarrow$\ref{44fxpro1}: Let $x\in A$. By Theorem \ref{mincor}\ref{mincor3} follows that $\mathscr{F}(x)\veebar x^{\perp}$ cannot be contained in any minimal prime filter and so it contains a dense element like $d$. Hence, there are $a\in \mathscr{F}(x)$ and $b\in x^{\perp}$ such that $a\odot b\leq d$. So for some integer $n$ follows that $x^{n}\odot b$ is dense. Let $u\in b^{\perp}$ and $v\in x^{\perp}$. Thus we have $(u\vee v)\vee b=1$ and $(u\vee v)\vee x^{n}=1$, and by using \ref{res2} we deduce that $(u\vee v)\vee (x^{n}\odot b)=1$. It shows that $u\vee v=1$ and it means that $b^{\perp}\subseteq x^{\perp\perp}$. The other inclusion is evident by Proposition \ref{1fxpro}\ref{1fxpro1}, and so the result holds.
\end{enumerate}
\end{proof}

Quasicomplemented residuated lattices are characterized under the name of $\star$-residuated lattices in \cite{rasdeh}. In the following theorem, we give a topological characterization for quasicomplemented residuated lattice.
\begin{theorem}\label{compmin}
Let $\mathfrak{A}$ be a residuated lattice. The following assertions are equivalent:
\begin{enumerate}
   \item [(1) \namedlabel{compmin1}{(1)}] $\mathfrak{A}$ is quasicomplemented;
   \item [(2) \namedlabel{compmin2}{(2)}] $\tau_{h}$ and $\tau_{d}$ coincide;
   \item [(3) \namedlabel{compmin3}{(3)}] $(Min(\mathfrak{A});\tau_{h})$ is compact.
 \end{enumerate}
\end{theorem}
\begin{proof}
It follows by \citet[Theorem 4.6]{rasdeh}.
\end{proof}
\section{Disjunctive residuated lattices}\label{sec4}

\cite{spe} introduced a certain class of distributive lattices with zero named disjunctive lattices. This notion has been discussed in semilattices by \cite{buc} and in commutative semigroups by \cite{kis}. \citet[Theorem 7.6]{cor} proved that if $\mathfrak{A}$ is a disjunctive normal lattice and $Max(\mathfrak{A})$, the space of maximal filters of $\mathfrak{A}$ with the hull-kernel topology, is a compact Hausdorff totally disconnected space, then $\mathfrak{A}$ is complementedly normal. Also, \citet[Proposition 2.3]{cor1} showed that a disjunctive normal lattice is dual isomorphic to its lattice of annulets. Actually, disjunctive lattices are themselves important in the study of annulets; information can be obtained by dualizing Banaschewski's results in \citep[Section 4]{ban}. In this section, we introduce and study notions of disjunctive and weakly disjunctive residuated lattice.

Let $\mathfrak{A}$ be a residuated lattice. We set $D(\mathfrak{A})=\{d(x)|x\in A\}$ and $H(\mathfrak{A})=\{h(x)|x\in A\}$. By \citet[Proposition 3.6 and 3.10]{rasdeh}, it follows that $(D(\mathfrak{A});\cap,\cup,d(1)=\emptyset,d(0)=Min(\mathfrak{A}))$ and $(H(\mathfrak{A});\cap,\cup,h(0)=\emptyset,h(1)=Min(\mathfrak{A}))$ are bounded lattices. Consider the following diagram;
\begin{figure}[h!]
\centering
  \begin{tikzpicture}
  \matrix (m) [matrix of math nodes,row sep=8em,column sep=8em,minimum width=2em]
  {
     A & \mathscr{PF}(\mathfrak{A})& D(\mathfrak{A}) \\
     {} & \gamma(\mathfrak{A}) &  H(\mathfrak{A}) \\};
  \path[-stealth]
    (m-1-1) edge node [above] {$f_1:x\mapsto \mathscr{F}(x)$} (m-1-2)
            edge node[sloped, above, midway] {$f_2:x\mapsto x^{\perp}$} (m-2-2)
    (m-1-2) edge node [above] {$f_3:\mathscr{F}(x)\mapsto d(x)$} (m-1-3)
            edge node [sloped,above, midway] {$f_4:\mathscr{F}(x)\mapsto x^{\perp}$} (m-2-2)
    (m-1-3) edge node [sloped,above, midway] {$f_5:d(x)\mapsto h(x)$} (m-2-3)
    (m-2-2) edge node [below] {$f_6:x^{\perp}\mapsto h(x)$} (m-2-3);
\end{tikzpicture}
\caption{} \label{fig:M1}
\end{figure}

Recalling that, if $\mathfrak{A}$ and $\mathfrak{B}$ are two algebras of a same type and $f:\mathfrak{A}\longrightarrow \mathfrak{B}$ is a homomorphism, then $\kappa(f)=\{(a_1,a_2)\in A^2|f(a_1)=f(a_2)\}$ is a congruence relation on $\mathfrak{A}$. The following remark has a routine verification.
\begin{remark}\label{diarem}
\begin{enumerate}
   \item [(1) \namedlabel{diarem1}{(1)}] By Proposition \ref{1fxpro}\ref{1fxpro4} and Corollary \ref{mindperp} follows that $f_3,f_4,f_5$ and $f_6$ are well-define.
   \item [(2) \namedlabel{diarem2}{(2)}] $f_2=f_4f_1$ and $f_5f_3=f_6f_4$.
   \item [(3) \namedlabel{diarem3}{(3)}] $f_2,f_3$ are lattice epimorphisms and $f_1,f_4$ are dual lattice epimorphisms.
   \item [(4) \namedlabel{diarem4}{(4)}] $f_5$ is a dual lattice isomorphism and $f_6$ is a lattice isomorphism.
   \item [(5) \namedlabel{diarem5}{(5)}] $\Re:=\kappa(f_3)=\kappa(f_4)$, so $f_3$ is injective if and only if $f_4$ is injective.
   \item [(6) \namedlabel{diarem6}{(6)}] $f_2$ is injective if and only if $f_1,f_4$ are injective.
   \item [(7) \namedlabel{diarem7}{(7)}] $\mathscr{F}(1)/\Re=\{\mathscr{F}(1)\}$ and $\mathscr{F}(0)/\Re=\{\mathscr{F}(x)|x\in \mathfrak{d(A)}\}$.
   \item [(8) \namedlabel{diarem8}{(8)}] $1/\kappa(f_2)=\{1\}$ and $0/\kappa(f_2)=\mathfrak{d(A)}$.
   \item [(9) \namedlabel{diarem9}{(9)}] $1/\kappa(f_1)=\{1\}$ and $0/\kappa(f_1)=N(\mathfrak{A})$.
 \end{enumerate}
\end{remark}
\begin{proposition}\label{equqaudis}
Let $\mathfrak{A}$ be a residuated lattice. The following assertions are equivalent:
\begin{enumerate}
   \item [(1) \namedlabel{equqaudis1}{(1)}] $\mathfrak{A}$ is quasicomplemented;
   \item [(2) \namedlabel{equqaudis2}{(2)}] $\ell(\mathfrak{A})/\kappa(f_2)$ is a Boolean lattice;
   \item [(3) \namedlabel{equqaudis3}{(3)}] $\mathscr{PF}(\mathfrak{A})/\Re$ is a Boolean lattice.
 \end{enumerate}
\end{proposition}
\begin{proof}
It is an immediate consequence of Proposition \ref{qucomeq} and  \textsc{Remark} \ref{diarem}.
\end{proof}

\begin{definition}
Let $\mathfrak{A}$ be a residuated lattice. With notations of Figure \ref{fig:M1}, $\mathfrak{A}$ is called \textit{disjunctive} provided that $f_2$ is injective and \textit{weakly disjunctive} provided that $f_3$ (or, equivalently $f_4$) is injective.
\end{definition}

By \textsc{Remark} \ref{diarem}\ref{diarem6}, it is evident that if a residuated lattice is disjunctive, then it is weakly disjunctive. In the following proposition the interrelation between the subclasses of quasicomplemented and disjunctive residuated lattices is given (See Fig. \ref{figsto}).
\begin{proposition}\label{equqaudisco}
Let $\mathfrak{A}$ be a residuated lattice. The following assertions are equivalent:
\begin{enumerate}
   \item [(1) \namedlabel{equqaudisco1}{(1)}] $\mathfrak{A}$ is quasicomplemented and disjunctive;
   \item [(2) \namedlabel{equqaudisco2}{(2)}] $\ell(\mathfrak{A})$ is a Boolean lattice;
   \item [(3) \namedlabel{equqaudisco3}{(3)}] $\mathfrak{A}$ is quasicomplemented, $\mathfrak{d(A)}=\{0\}$ and the operation $\neg$ is injective as a function.
 \end{enumerate}
\end{proposition}
\begin{proof}
\item [\ref{equqaudisco1}$\Rightarrow$\ref{equqaudisco2}:] It follows by Proposition \ref{equqaudis}.
\item [\ref{equqaudisco2}$\Rightarrow$\ref{equqaudisco3}:] Applying Proposition \ref{1propos}, it follows that the operation $\neg$ is injective as a function. By Remark \ref{genfilprop}\ref{genfilprop6} and Proposition \ref{4fxpro}\ref{4fxpro5}, it follows that $\mathfrak{d(A)}=\{0\}$. Also, by Proposition \ref{qucomeq} follows that $\mathfrak{A}$ is quasicomplemented.
\item [\ref{equqaudisco3}$\Rightarrow$\ref{equqaudisco1}:] Let $x^{\perp}=y^{\perp}$. So we have $(x\odot \neg y)^{\perp}=\{1\}$ and it implies that $\neg y\leq \neg x$. Analogously, we can conclude that $\neg x\leq \neg y$ and it implies that $\neg x=\neg y$. Since $\neg$ is an injective operation so the result holds.
\end{proof}
\begin{remark}
  According to \cite[Corollary 3.2]{ciu1} follows that a residuated lattice $\mathfrak{A}$ is Boolean if and only if for any $a\in A$ we have $a\wedge \neg a=0$ and $a\vee \neg a=1$. It gives a new characterisation for quasicomplemented and disjunctive residuated lattices.
\end{remark}
\begin{figure}[h!]
\centering
\begin{tikzpicture}[rounded corners=3pt]

\draw (0,0) rectangle (10,5);
\draw (1,2) rectangle (8,4);
\draw (2,1) rectangle (9,3);

\draw (1.7,4.75) node {\small{Residuated lattices}};
\draw (4.5,3.75) node {\small{Quasicomplemented residuated lattices}};
\draw (5,2.5) node {\small{Boolean lattices}};
\draw (6.5,1.25) node {Disjunctive residuated lattices};
\end{tikzpicture}
\caption{} \label{figsto}
\end{figure}
\begin{proposition}\label{equqauwdisco}
Let $\mathfrak{A}$ be a residuated lattice. The following assertions are equivalent:
\begin{enumerate}
   \item [(1) \namedlabel{equqauwdisco1}{(1)}] $\mathfrak{A}$ is quasicomplemented and weakly disjunctive;
   \item [(2) \namedlabel{equqauwdisco2}{(2)}] $\mathscr{PF}(\mathfrak{A})$ is a Boolean lattice.
 \end{enumerate}
\end{proposition}
\begin{proof}
\item [\ref{equqauwdisco1}$\Rightarrow$\ref{equqauwdisco2}:] It follows by Proposition \ref{equqaudis}.
\item [\ref{equqauwdisco2}$\Rightarrow$\ref{equqauwdisco1}:] Let $(\mathscr{F}(x))^{c}=\mathscr{F}(y)$. So we have $\mathscr{F}(x)\cap \mathscr{F}(y)=1$ and $\mathscr{F}(x)\veebar \mathscr{F}(y)=A$. The former implies $\mathscr{F}(y)\subseteq x^{\perp}$ and so the latter implies $\mathscr{F}(x)\veebar x^{\perp}=A$. It shows that $\mathscr{F}(y)=x^{\perp}$. So $\mathfrak{A}$ is quasicomplemented since $x^{\perp\perp}=y^{\perp}$ and $\mathfrak{A}$ is weakly disjunctive since $x^{\perp}=y^{\perp}$ implies $(\mathscr{F}(x))^{c}=(\mathscr{F}(y))^{c}$ and so $\mathscr{F}(x)=\mathscr{F}(y)$.
\end{proof}
\begin{remark}
  Let $\mathfrak{A}$ be a residuated lattice. Applying Proposition \ref{genfilprop}, it is easy to see that $\mathscr{PF}(\mathfrak{A})$ is a Boolean lattice if and only if for any $a\in A$ there exists $b\in a^{\perp}$ such that $a\odot b\in N(\mathfrak{A})$. It gives a new characterisation for quasicomplemented and weakly disjunctive residuated lattices.
\end{remark}
\section{$\alpha$-filters}\label{sec5}

 The notion of $\alpha$-ideals introduced by \cite{cor1} in distributive lattice with $0$. \cite{jay} generalized the concept of $\alpha$-ideals to $0$-distributive lattices. Some further properties of $\alpha$-ideals for $0$-distributive lattices were obtained by \cite{paw,paw1}. \cite{hav} proposed the concept of $\alpha$-filters in BL-algebras as the dual notion of $\alpha$-ideals. Recently, \cite{don} extend the concept of $\alpha$-filters to residuated lattices. In this section we introduce and study the notion of $\alpha$-filter in residuated lattices.
\begin{definition}
Let $\mathfrak{A}$ be a residuated lattice. A filter $F$ of $\mathfrak{A}$ is called an $\alpha$-\textit{filter} if for any $x\in F$ we have $x^{\perp\perp}\subseteq F$. The set of $\alpha$-filters of $\mathfrak{A}$ is denoted by $\alpha(\mathfrak{A})$. It is obvious that $\{1\},A\in \alpha(\mathfrak{A})$.
\end{definition}
\begin{example}\label{pppfex0}
Consider the residuated lattice $\mathfrak{A}_7$ from Example \ref{rex1}. With notations of Example \ref{fex1}, $F_2$ and $F_3$ are $\alpha$-filters of $\mathfrak{A}_7$.
\end{example}

Let $\mathfrak{A}$ be a residuated lattice. It is obvious that $\alpha(\mathfrak{A})$ is an algebraic closed set system on $\mathfrak{A}$. The closure operator associated with this closed set system is denoted by $\alpha^{\mathfrak{A}}:\mathcal{P}(A)\longrightarrow \mathcal{P}(A)$. Thus for any subset $X$ of $A$, $\alpha^{\mathfrak{A}}(X)=\cap\{F\in \alpha(\mathfrak{A})|X\subseteq F\}$ is the smallest $\alpha$-filter of $\mathfrak{A}$ contains $X$ which it is called the \textit{$\alpha$-filter of $\mathfrak{A}$ generated by $X$}. When there is no ambiguity we will drop the superscript $\mathfrak{A}$. Hence $\alpha(\mathfrak{A})$ is a complete compactly generated lattice where the infimum is the set-theoretic intersection and the supremum of $\mathcal{F}\subseteq \alpha(\mathfrak{A})$ is $\vee^{\alpha} \mathcal{F}=\alpha(\cup \mathcal{F})$. It is obvious that $\alpha(X)=\alpha(\mathscr{F}(X))$ for any $X\subseteq A$ and so we have $\vee^{\alpha} \mathcal{F}=\alpha(\veebar \mathcal{F})$.
\begin{proposition}\label{framealpha}
For any residuated lattice $\mathfrak{A}$, $(\alpha(\mathfrak{A});\cap,\vee^{\alpha})$ is a frame.
\end{proposition}
\begin{proof}
We know that $\alpha(\mathfrak{A})$ is a complete lattice. Let $\{F\}\cup \mathcal{G}$ be a family of $\alpha$-filters. We have the following sequence of formulas:
\[
\begin{array}{ll}
  F\cap(\vee^{\alpha}\mathcal{G}) & =\alpha(F\cap(\veebar\mathcal{G})) \\
   & =\alpha(\veebar_{G\in \mathcal{G}}(F\cap G)) \\
   & =\vee^{\alpha}_{G\in \mathcal{G}}(F\cap G).
\end{array}
\]
It shows that $\alpha(\mathfrak{A})$ is a frame.
\end{proof}

It is well-known that a lattice is a frame if and only if it is a complete Heyting algebra. So due to Proposition \ref{framealpha}, we deduce that for any residuated lattice $\mathfrak{A}$, $(\alpha(\mathfrak{A});\cap,\vee^{\alpha},\hookrightarrow,\textbf{1},A)$ is a Heyting algebra, where $F\hookrightarrow G=\veebar\{H\in \alpha(\mathfrak{A})|F\cap H\subseteq G\}$ for any $F,G\in \alpha(\mathfrak{A})$.

\begin{proposition}\label{veealphafilter}
Let $\{F\}\cup \{F_i\}_{i\in I}$ be a family of filters in a residuated lattice $\mathfrak{A}$ and $x,y\in A$. The following assertions hold:
\begin{enumerate}
  \item [$(1)$ \namedlabel{veealphafilter1}{$(1)$}] $\alpha(F)=\cup_{x\in F} x^{\perp\perp}$;
  \item [$(2)$ \namedlabel{veealphafilter2}{$(2)$}] $\vee^{\alpha}_{i\in I}F_i=\{a\in A|a\in (f_{i_1}\odot\cdots\odot f_{i_n})^{\perp\perp},~\exists n\in \mathbb{N},~i^{n}_{1}\in I,~f_{i_j}\in F_{i_j}\}$;
  \item [$(3)$ \namedlabel{veealphafilter3}{$(3)$}] $\alpha(F,x):=\alpha(F\cup \{x\})=\cup_{f\in F,n\in \mathbb{N}}(f\odot x^n)^{\perp\perp}$;
  \item [$(4)$ \namedlabel{veealphafilter4}{$(4)$}] $\alpha(F,x)\cap \alpha(F,y)=\alpha(F,x\vee y)$;
  \item [$(5)$ \namedlabel{veealphafilter5}{$(5)$}] $\alpha(F,x)\vee^{\alpha} \alpha(F,y)=\alpha(F,x\odot y)$.
\end{enumerate}
\end{proposition}
\begin{proof}
\item [\ref{veealphafilter1}:] It proves quite in a routine way.
\item [\ref{veealphafilter2}:] We have $\vee^{\alpha}_{i\in I}F_i=\alpha(\veebar _{i\in I}F_i)=\{a\in A|a\in x^{\perp\perp},~for~some~x\in \veebar _{i\in I}F_i\}$. Let $\Sigma=\{a\in A|a\in (\odot_{t=i_1}^{i_n}f_{t})^{\perp\perp},\exists n,i_1,\cdots,i_n\in \mathbb{N}~f_{i_j}\in F_{i_j}\}$. It is obvious that $\Sigma \subseteq \vee^{\alpha}_{i\in I}F_i$. Suppose that $a\in \vee^{\alpha}_{i\in I}F_i$. Thus $a\in x^{\perp\perp}$ for some $x\in \veebar _{i\in I}F_i$. So there exist an integer $n$, $i_1,\cdots,i_n\in I$ and $f_{i_j}\in F_{i_j}$ such that $f_{i_1}\odot\cdots\odot f_{i_n}\leq x$. By Proposition \ref{4fxpro}\ref{4fxpro1} follows that $x^{\perp\perp}\subseteq (f_{i_1}\odot\cdots\odot f_{i_n})^{\perp\perp}$ and it means $a\in \Sigma$.
\item [\ref{veealphafilter3}:] It follows by \ref{veealphafilter1}.
\item [\ref{veealphafilter4}:] By Remark \ref{genfilprop}\ref{genfilprop3}, we have the following formulas:
 \[
 \begin{array}{ll}
    \alpha(F,x)\cap \alpha(F,y)& =\alpha(\mathscr{F}(F,x))\cap \alpha(\mathscr{F}(F,y)) \\
    & =\alpha(\mathscr{F}(F,x)\cap \mathscr{F}(F,y)) \\
    & =\alpha(\mathscr{F}(F,x\vee y)\\
    & =\alpha(F,x\vee y).
 \end{array}
 \]
 \item [\ref{veealphafilter5}:] By Remark \ref{genfilprop}\ref{genfilprop4}, we have the following formulas:
 \[
 \begin{array}{ll}
    \alpha(F,x)\vee^{\alpha} \alpha(F,y)& =\alpha(\mathscr{F}(F,x))\vee^{\alpha} \alpha(\mathscr{F}(F,y)) \\
    & =\alpha(\mathscr{F}(F,x)\veebar \mathscr{F}(F,y)) \\
    & =\alpha(\mathscr{F}(F,x\odot y))\\
    & =\alpha(F,x\odot y).
 \end{array}
 \]
\end{proof}

In the following proposition, we give some equivalent assertions for a filter to be an $\alpha$-filter.
\begin{proposition}\label{alpfiltpro}
Let $\mathfrak{A}$ be a residuated lattice and $F$ be a filter of $\mathfrak{A}$. The following assertions are equivalent:
\begin{enumerate}
  \item [$(1)$ \namedlabel{alpfiltpro1}{$(1)$}] $F$ is an $\alpha$-filter;
  \item [$(2)$ \namedlabel{alpfiltpro2}{$(2)$}] if $x^{\perp}=y^{\perp}$ and $x\in F$, then $y\in F$ for any $x,y\in A$;
  \item [$(3)$ \namedlabel{alpfiltpro3}{$(3)$}] if $d(x)=d(y)$ and $x\in F$, then $y\in F$ for any $x,y\in A$;
  \item [$(4)$ \namedlabel{alpfiltpro4}{$(4)$}] if $h(x)=h(y)$ and $x\in F$, then $y\in F$ for any $x,y\in A$.
\end{enumerate}
\end{proposition}
\begin{proof} By Corollary \ref{mindperp}, it follows that \ref{alpfiltpro1}, \ref{alpfiltpro2} and \ref{alpfiltpro3} are equivalent. So we only prove the other cases.
\item [] \ref{alpfiltpro1}$\Rightarrow$\ref{alpfiltpro2}: Let $x^{\perp}=y^{\perp}$ for some $y\in A$ and $x\in F$. Then $y\in y^{\perp\perp}=x^{\perp\perp}\subseteq F$.
\item [] \ref{alpfiltpro2}$\Rightarrow$\ref{alpfiltpro1}: Let $x\in F$ and $y\in x^{\perp\perp}$. So $x^{\perp}\subseteq y^{\perp}$. By Proposition \ref{4fxpro}\ref{4fxpro4} follows that $y^{\perp}=x^{\perp}\vee^{\Gamma} y^{\perp}=(x\vee y)^{\perp}$ and it states that $y\in F$ since $x\vee y\in F$.
\end{proof}

Let $\mathscr{A}=(A;\leq)$ and $\mathscr{B}=(B;\preccurlyeq)$ be two posts. We recall that a pair $(f,g)$ is called a \textit{an Adjunction (or isotone Galois connection)} between posets $\mathscr{A}$ and $\mathscr{B}$, where $f:A\longrightarrow B$ and $g:B\longrightarrow A$ are two functions such that for all $a\in A$ and $b\in B$, $f(a)\leq b$ if and only if $a\preccurlyeq g(b)$. It is well known that $(f,g)$ is an adjunction connection if and only if $gf$ is inflationary, $fg$ is deflationary and $f,g$ are isotone \citep[Theorem 2]{gar}. It is well-known $\mathscr{C}_{gf}=g(B)$, where $\mathscr{C}_{gf}$ is the set of fixed point of the closure operator $gf$.
\begin{theorem}\label{filtergammafilt}
Let $\mathfrak{A}$ be a residuated lattice. We define
\[
\begin{array}{ll}
\begin{array}{cclc}
  \Phi:&\mathscr{F}(\mathfrak{A})&\longrightarrow&\mathscr{F}(\gamma(\mathfrak{A}))\\
   & F&\longmapsto& \{x^{\perp}|x\in F\},
\end{array} &
\begin{array}{cclc}
  \Psi:&\mathscr{F}(\gamma(\mathfrak{A}))&\longrightarrow&\mathscr{F}(\mathfrak{A})\\
   & F&\longmapsto& \{x\in A|x^{\perp}\in F\}.
\end{array}
\end{array}
\]
Then the pair $(\Phi,\Psi)$ is an adjunction and we have $\Psi\Phi(F)=\alpha(F)$ for any $F\in \mathscr{F}(\mathfrak{A})$. In particular we have $\mathscr{C}_{\Psi\Phi}=\alpha(\mathfrak{A})$.
\end{theorem}
\begin{proof}
Quite in a routine way we can show that the pair $(\Phi,\Psi)$ forms an adjunction. Let $F$ be a filter of $\mathfrak{A}$. We have the following formulas:
\[
\begin{array}{ll}
  \Psi\Phi(F) & =\Psi(\Phi(F)) \\
   & =\{a\in A|a^{\perp}\in \Phi(F)\} \\
   & =\{a\in A|a^{\perp}=x^{\perp}, \exists x\in F\} \\
   & =\{a\in A|a\in x^{\perp\perp}, \exists x\in F\}=\cup_{x\in F}x^{\perp\perp}=\alpha(F). \\
\end{array}
\]
The rest is evident.
\end{proof}

The next theorem should be compared with Theorem \ref{prfilth}.
\begin{theorem}\label{alphaprfilth}
Let $\mathscr{C}$ be a $\vee$-closed subset of $\mathfrak{A}$ which does not meet the $\alpha$-filter $F$. Then $F$ is contained in an $\alpha$-filter $P$ which is maximal with respect to the property of not meeting $\mathscr{C}$; furthermore $P$ is prime.
\end{theorem}
\begin{proof}
Let $\Sigma=\{G\in \alpha(\mathfrak{A})|F\subseteq G,~G\cap \mathscr{C}=\emptyset\}$. It is easy to see that $\Sigma$ satisfies the conditions of Zorn's lemma. Let $P$ be a maximal element of $\Sigma$. Assume that $x\vee y\in P$ and neither $x\notin P$ nor $y\notin P$. By maximality of $P$ we have $\alpha(P,x)\cap \mathscr{C}\neq\emptyset$ and $\alpha(P,y)\cap \mathscr{C}\neq\emptyset$. Suppose $a_x\in \alpha(P,x)\cap \mathscr{C}$ and $a_y\in \alpha(P,y)\cap \mathscr{C}$. By Proposition \ref{veealphafilter}\ref{veealphafilter3}, there exist $p_x,p_y\in P$ and integers $n,m$ such that $a_x\in (p_x\odot x^n)^{\perp\perp}$ and $a_y\in (p_y\odot y^m)^{\perp\perp}$. It follows that
 \[
 \begin{array}{ll}
   a_x\vee a_y & \in (p_x\odot x^n)^{\perp\perp}\cap (p_y\odot y^m)^{\perp\perp} \\
    & =((p_x\odot x^n)\vee (p_y\odot y^m))^{\perp\perp} \\
    & \subseteq ((p_x\vee p_y)\odot (p_x\vee y^{m}))\odot (x^{n}\vee p_y)\odot (x\vee y)^{nm})^{\perp\perp} \\
    & \subseteq P.
 \end{array}
 \]
 It is a contradiction. So $x\in P$ or $y\in P$ and it shows that $P$ is a prime filter.
\end{proof}

In the sequel for any residuated lattice $\mathfrak{A}$ we set $Spec_{\alpha}(\mathfrak{A})=Spec(\mathfrak{A})\cap \alpha(\mathfrak{A})$. \begin{corollary}\label{alphaintprimfilt}
Let $F$ be an $\alpha$-filter of $\mathfrak{A}$ and $X$ be a non-empty subset of $A$. The following assertions hold:
\begin{enumerate}
\item  [(1) \namedlabel{alphaintprimfilt1}{(1)}]  If $X\nsubseteq F$, then there exists $P\in Spec_{\alpha}(\mathfrak{A})$ such that $F\subseteq P$ and $P$ is maximal with respect to the property $X\nsubseteq P$;
\item  [(2) \namedlabel{alphaintprimfilt2}{(2)}] $\alpha(X)=\bigcap \{P\in Spec_{\alpha}(\mathfrak{A})|X\subseteq P\}$.
\end{enumerate}
\end{corollary}
\begin{proof}
\begin{enumerate}
  \item [\ref{alphaintprimfilt1}:] Let $x\in X-F$. By taking $\mathscr{C}=\{x\}$ it follows by Theorem \ref{alphaprfilth}.
  \item [\ref{alphaintprimfilt2}:] Set $\sigma_{X}=\{P\in Spec_{\alpha}(\mathfrak{A})|X\subseteq P\}$. Obviously, we have $\alpha(X)\subseteq \bigcap \sigma_{X}$. Now, let $a\notin \alpha(X)$. By \ref{alphaintprimfilt1} follows that there exits an $\alpha$-prime filter $P$ containing $\alpha(X)$ such that $a\notin P$. It shows that $a\notin \bigcap \sigma_{X}$.
\end{enumerate}
\end{proof}

In the following proposition we characterize weakly disjunctive residuated lattices by means of $\alpha$ filters.
\begin{proposition}\label{alphadiseq}
Let $\mathfrak{A}$ be a residuated lattice. The following assertions are equivalent:
\begin{enumerate}
\item  [$(1)$ \namedlabel{alphadiseq1}{$(1)$}] $\mathfrak{A}$ is weakly disjunctive;
\item  [$(2)$ \namedlabel{alphadiseq2}{$(2)$}] any filter of $\mathfrak{A}$ is an $\alpha$-filter;
\item  [$(3)$ \namedlabel{alphadiseq3}{$(3)$}] any prime filter of $\mathfrak{A}$ is an $\alpha$-filter.
\end{enumerate}
\end{proposition}
\begin{proof}
\item [\ref{alphadiseq1}$\Rightarrow$\ref{alphadiseq2}:] It follows by Proposition \ref{alpfiltpro}\ref{alpfiltpro2}.
\item [\ref{alphadiseq2}$\Rightarrow$\ref{alphadiseq3}:] It is trivial.
\item [\ref{alphadiseq3}$\Rightarrow$\ref{alphadiseq1}:] Let $x^{\perp}=y^{\perp}$ and $\mathscr{F}(x)\neq\mathscr{F}(y)$ for some $x,y\in A$. Without loss of generality suppose that $\mathscr{F}(x)\nsubseteq \mathscr{F}(y)$. Let $\Sigma=\{F\in \mathscr{F}(\mathfrak{A})|y\in F,~x\notin F\}$. It is obvious that $\mathscr{F}(y)\in \Sigma$ and $\Sigma$ satisfies the conditions of Zorn's lemma. Let $P$ be the maximal element of $\Sigma$. Let $a\vee b\in P$ and $a,b\notin P$. So $\mathscr{F}(P,a),\mathscr{F}(P,b)\notin \Sigma$ and it states that $x\in \mathscr{F}(P,a)\cap \mathscr{F}(P,b)$. By Remark \ref{genfilprop}\ref{genfilprop3} follows that $x\in \mathscr{F}(P,a)\cap \mathscr{F}(P,b)=\mathscr{F}(P,a\vee b)=P$ and it leads us to a contradiction. Thus $P$ is a prime filter and it implies $x\in x^{\perp\perp}=y^{\perp\perp}\subseteq P$; a contradiction.
\end{proof}

In the following proposition we show that any coannihilator filter and any $\omega$-filter is an $\alpha$-filter.
\begin{proposition}\label{alphafofilter}
Let $\mathfrak{A}$ be a residuated lattice. The following assertions hold:
\begin{enumerate}
  \item [$(1)$ \namedlabel{alphafofilter1}{$(1)$}] $\Gamma(\mathfrak{A})\subseteq \alpha(\mathfrak{A})$;
  \item [$(2)$ \namedlabel{alphafofilter2}{$(2)$}] $\Omega(\mathfrak{A})\subseteq \alpha(\mathfrak{A})$.
\end{enumerate}
\end{proposition}
\begin{proof}
\item [\ref{alphafofilter1}:] Let $F$ be a coannihilator filter and $x\in F$. So $x^{\perp\perp}\subseteq F^{\perp\perp}=F$ and it shows that $F$ is an $\alpha$-filter.
\item [\ref{alphafofilter2}:] Let $F$ be an $\omega$-filter of $\mathfrak{A}$. So there exists a lattice ideal $I$ of $\mathfrak{A}$ such that $F=\omega(I)$. Consider $y\in F$. So there exists $x\in I$ such that $y\in x^{\perp}$ and it follows that $x\in y^{\perp}$. Assume  $z\in y^{\perp\perp}$. Thus $y^{\perp}\subseteq z^{\perp}$ and it follows that $x\in z^{\perp}$. So we observe that $z\in x^{\perp}\subseteq F$.
\end{proof}
\begin{corollary}\label{examomegfilt}
Let $\mathfrak{A}$ be a residuated lattice. The following assertions hold:
\begin{enumerate}
  \item [$(1)$ \namedlabel{examomegfilt1}{$(1)$}] Any non-dense prime filter is an $\alpha$-filter;
  \item [$(2)$ \namedlabel{examomegfilt2}{$(2)$}] any minimal prime filter is an $\alpha$-filter.
\end{enumerate}
\end{corollary}
\begin{proof}
\item [\ref{examomegfilt1}]: It is an immediate consequence of Proposition \ref{nondepri} and Proposition \ref{alphafofilter}\ref{alphafofilter1}.
\item [\ref{examomegfilt2}]: It is an immediate consequence of Theorem \ref{mincor} and Proposition \ref{alphafofilter}\ref{alphafofilter2}.
\end{proof}

In the following proposition we give some equivalent conditions for each $\alpha$-filter to be a coannihilator.
\begin{proposition}\label{alphacoann}
Let $\mathfrak{A}$ be a residuated lattice. The following assertions are equivalent:
\begin{enumerate}
  \item [$(1)$ \namedlabel{alphacoann1}{$(1)$}] Any proper $\alpha$-filter is non-dense;
  \item [$(2)$ \namedlabel{alphacoann2}{$(2)$}] any dense filter contains a dense element;
  \item [$(3)$ \namedlabel{alphacoann3}{$(3)$}] any $\alpha$-filter is a coannihilator;
  \item [$(4)$ \namedlabel{alphacoann4}{$(4)$}] for any proper $\alpha$-filter $F$ there exists a proper $\alpha$-filter $G$ such that $F\cap G=\{1\}$ ($\alpha(\mathfrak{A})$ is semi-complemented);
  \item [$(5)$ \namedlabel{alphacoann5}{$(5)$}] $\alpha(\mathfrak{A})$ has a unique dense element.
\end{enumerate}
Moreover, any of the above assertions implies that $\mathfrak{A}$ is quasicomplemented.
\end{proposition}
\begin{proof}
\begin{enumerate}
  \item [] \ref{alphacoann1}$\Rightarrow$\ref{alphacoann2}: Let $F$ be a dense filter of $\mathfrak{A}$. It implies that $\alpha(F)$ is a dense filter and it means that $\alpha(F)=A$. Therefore $0\in \alpha(F)$ and it means that $x^{\perp}=1$ for some element $x\in F$.
  \item [] \ref{alphacoann2}$\Rightarrow$\ref{alphacoann3}: Let $F$ be an $\alpha$-filter. $F,F^{\perp}\subseteq F\veebar F^{\perp}$ and so $(F\veebar F^{\perp})^{\perp}\subseteq F^{\perp}\cap F^{\perp\perp}=\textbf{1}$. So $F\veebar F^{\perp}$ is a dense filter and so it contains a dense element. Assume that $x$ is a dense element in $F\veebar F^{\perp}$. So there exist $a\in F$ and $b\in F^{\perp}$ such that $a\odot b\leq x$. It implies that $a^{\perp}\cap b^{\perp}\subseteq x^{\perp}=\textbf{1}$ and since $a^{\perp\perp}$ is the pseudo-complement of $a^{\perp}$ it follows that $b^{\perp}\subseteq a^{\perp\perp}\subseteq F$ as $F$ is an $\alpha$-filter. On the other hand $b\in F^{\perp}$ implies that $F^{\perp\perp}\subseteq b^{\perp}$ and it states that $F^{\perp\perp}\subseteq F$. The other inclusion is evident.
  \item [] \ref{alphacoann3}$\Rightarrow$\ref{alphacoann4}: Since the set of all coannihilators forms a Boolean lattice, it follows that $\alpha(\mathfrak{A})$ is semi-complemented.
  \item [] \ref{alphacoann4}$\Rightarrow$\ref{alphacoann5}: It is obvious that $A$ is a dense element of $\alpha(\mathfrak{A})$. Now, if $F$ is a proper dense $\alpha$-filter, then there exists a filter $G\neq \textbf{1}$ such that $F\cap G=\textbf{1}$ and it implies that $G\subseteq F^{\perp}=\textbf{1}$. So $G=\textbf{1}$ and it is a contradiction.
  \item [] \ref{alphacoann5}$\Rightarrow$\ref{alphacoann1}: It is trivial.

 Now, let $\mathfrak{A}$ satisfies \ref{alphacoann1} and $x\in A$. Set $F=x^{\perp}\veebar x^{\perp\perp}$. It implies that $\alpha(F)^{\perp}\subseteq x^{\perp}\cap x^{\perp\perp}=\textbf{1}$ and it means that $\alpha(F)$ is a dense filter. So $\alpha(F)=A$ and it states that $0\in \alpha(F)$. It follows that $y^{\perp}=1$ for some element $y\in F$. Hence there exist $a\in x^{\perp}$ and $b\in x^{\perp\perp}$ such that $a\odot b\leq y$. Therefore $a^{\perp}\cap b^{\perp}=\textbf{1}$ and it shows that $a^{\perp}\subseteq b^{\perp\perp}\subseteq x^{\perp\perp}$. On the other hand $a\in x^{\perp}$ gives $x^{\perp\perp}\subseteq a^{\perp}$. Combining both the inclusions follows that $x^{\perp\perp}=a^{\perp}$. Hence $\mathfrak{A}$ is quasicomplemented.
\end{enumerate}
\end{proof}

In the following proposition we give some equivalent conditions for each $\alpha$-filter to be an $\omega$-filter.
\begin{theorem}\label{quacomofli}
Let $\mathfrak{A}$ be a residuated lattice. The following assertions are equivalent:
\begin{enumerate}
\item  [$(1)$ \namedlabel{quacomofli1}{$(1)$}] $\mathfrak{A}$ is quasicomplemented;
\item  [$(2)$ \namedlabel{quacomofli2}{$(2)$}] every $\alpha$-filter is an $\omega$-filter;
\item  [$(3)$ \namedlabel{quacomofli3}{$(3)$}] every coannihilator is an $\omega$-filter;
\item  [$(4)$ \namedlabel{quacomofli4}{$(4)$}] For any $x\in A$, $x^{\perp\perp}$ is an $\omega$-filter.
\end{enumerate}
\end{theorem}
\begin{proof}
\begin{enumerate}
\item  []\ref{quacomofli1}$\Rightarrow$\ref{quacomofli2}: Let $F$ be an $\alpha$-filter and set $I_{F}=\{a\in A|f^{\perp}\subseteq a^{\perp\perp}, for~some~f\in F\}$. It is obvious that $0\in I_{F}$. Let $a_1,a_2\in I_{F}$. So $f_{1}^{\perp}\subseteq a_{1}^{\perp\perp}$ and $f_{2}^{\perp}\subseteq a_{2}^{\perp\perp}$ for some $f_1,f_2\in F$. By \textsc{Remark} \ref{4fxpro} follows that $(f_1\odot f_2)^{\perp}\subseteq (a_1\vee a_2)^{\perp\perp}$ and it states that $a_1\vee a_2\in I_{F}$. Now let $a_1\leq a_2$ and $a_2\in I_{F}$. So there exists $f\in F$ such that $f^{\perp}\subseteq a^{\perp\perp}_{2}$ and on the other hand we have $a^{\perp\perp}_{2}\leq a^{\perp\perp}_{1}$. Hence we have $f^{\perp}\subseteq a^{\perp\perp}_{1}$ and it means that $a_1\in I_{F}$. Thus $I_F$ is a lattice ideal of $\mathfrak{A}$. Now, let $x\in \omega(I_{F})$. Consequently, $x\in a^{\perp}$ for some $a\in I_{F}$ and so $x\in f^{\perp\perp}$ for some $f\in F$. Since $F$ is an $\alpha$-filter follows that $x\in F$. Otherwise, let $x\in F$. Since $\mathfrak{A}$ quasicomplemented follows that $x^{\perp\perp}=a^{\perp}$ for some $a\in A$. It means that $a\in I_{F}$. Therefore $x\in x^{\perp\perp}=a^{\perp}\subseteq \omega(I_{F})$ and it shows that $F\subseteq \omega(I_{F})$.
\item  []\ref{quacomofli2}$\Rightarrow$\ref{quacomofli3}: It is obvious, since any coannihilator is an $\alpha$-filter.
\item  []\ref{quacomofli3}$\Rightarrow$\ref{quacomofli4}: It is obvious, since for any $x\in A$, $x^{\perp\perp}$ is a coannihilator.
\item  []\ref{quacomofli4}$\Rightarrow$\ref{quacomofli1}: Let $x\in A$. So there exists a lattice ideal $I$ such that $x^{\perp\perp}=\omega(I)$. So $x\in \omega(I)$ and it implies that $x\in y^{\perp}$ for some $y\in I$. Thus $x^{\perp\perp}\subseteq y^{\perp}\subseteq \omega(I)=x^{\perp\perp}$.
\end{enumerate}
\end{proof}
\begin{lemma}\label{propqunode}
Let $F$ be a filter of a residuated lattice $\mathfrak{A}$. Then $\alpha(F)$ is proper if and only if $F$ contains no dense element.
\end{lemma}
\begin{proof}
Let $F\cap \mathfrak{d(A)}\neq \emptyset$. So there exists a dense element $d$ in $F$. By Proposition \ref{veealphafilter}\ref{veealphafilter1} we get that $A=d^{\perp\perp}\subseteq \alpha(F)$ and it implies $\alpha(F)=A$. Otherwise, if $\alpha(F)=A$ then $0\in \alpha(F)$. So there exists $x\in F$ such that $0\in x^{\perp\perp}$. Therefore $x^{\perp}\subseteq 0^{\perp}=\{1\}$ and it means that $x$ is a dense element.
\end{proof}
\begin{lemma}\label{primealpha}
Let $\mathfrak{A}$ be a quasicomplemented residuated lattice and $P$ be a prime filter. The following assertions are equivalent:
\begin{enumerate}
\item  [$(1)$ \namedlabel{primealpha1}{$(1)$}] $P$ is an $\alpha$-filter;
\item  [$(2)$ \namedlabel{primealpha2}{$(2)$}] $P$ contains no dense element;
\item  [$(3)$ \namedlabel{primealpha3}{$(3)$}] $P$ is minimal prime;
\item  [$(4)$ \namedlabel{primealpha4}{$(4)$}] $P$ contains precisely one of $x,y\in A$ such that $x^{\perp\perp}=y^{\perp}$.
\end{enumerate}
\end{lemma}
\begin{proof}
\ref{primealpha1}$\Rightarrow$\ref{primealpha2} follows by Lemma \ref{propqunode} and \ref{primealpha2}$\Rightarrow$\ref{primealpha3} follows by Proposition \ref{44fxpro}\ref{44fxpro2}.
\item []\ref{primealpha3}$\Rightarrow$\ref{primealpha4}: Let $x$ and $y$ be any pair of elements for which $x^{\perp\perp}=y^{\perp}$. Since $x\vee y=1$ and $P$ is prime so either $x\in P$ or $y\in P$. Also, $x^{\perp}\cap y^{\perp}=\{1\}$ implies that $x^{\perp}\subseteq P$ or $y^{\perp}\subseteq P$. Now we can deduce the result by Theorem \ref{mincor}\ref{mincor3}.
\item []\ref{primealpha4}$\Rightarrow$\ref{primealpha1}: Let $x\in P$ and $y$ be an element such that $x^{\perp\perp}=y^{\perp}$. Since $y\notin P$ so $y^{\perp}\subseteq P$ and it shows that $P$ is an $\alpha$-filter.
\end{proof}

In the next corollary for which we characterize quasicomplemented residuated lattice in terms of $\alpha$-filters should be compared with Proposition \ref{44fxpro}.
\begin{corollary}\label{45fxpro}
Let $\mathfrak{A}$ be a residuated lattice. The following assertions are equivalent:
\begin{enumerate}
\item  [$(1)$ \namedlabel{45fxpro1}{$(1)$}] $\mathfrak{A}$ is quasicomplemented;
\item  [$(2)$ \namedlabel{45fxpro2}{$(2)$}] any prime $\alpha$-filter is minimal prime;
\item  [$(3)$ \namedlabel{45fxpro3}{$(3)$}] any proper $\alpha$-filter is the intersection of minimal prime filters;
\item  [$(4)$ \namedlabel{45fxpro4}{$(4)$}] any proper $\alpha$-filter is contained in a minimal
prime filter.
\end{enumerate}
\end{corollary}
\begin{proof}
\ref{45fxpro1}$\Rightarrow$\ref{45fxpro2} is followed by Proposition \ref{44fxpro} and Lemma \ref{primealpha}, \ref{45fxpro2}$\Rightarrow$\ref{45fxpro3} is followed by Corollary \ref{alphaintprimfilt}\ref{alphaintprimfilt2} and \ref{45fxpro3}$\Rightarrow$\ref{45fxpro4} is obvious.
\item []\ref{45fxpro4}$\Rightarrow$\ref{45fxpro1}: Let $F$ be a filter such that $F\cap \mathfrak{d(A)}=\emptyset$. By Lemma \ref{propqunode}, $\alpha(F)$ is proper and so it is contained in a minimal prime filter. It means that $F$ is contained in a minimal prime filter. Hence, the result is followed by Proposition \ref{44fxpro}\ref{44fxpro3}.
\end{proof}


We end this paper by deriving a set of equivalent assertions for any $\alpha$-filter of a residuated lattice to become principal.
\begin{lemma}\label{prielepri}
Let $\mathfrak{A}$ be a residuated lattice, $F$ be an $\alpha$-filter and $a\in F$. If for any prime $\alpha$-filter $P$ which contains $a$ we have $F\subseteq P$, then $F=\alpha(a)$.
\end{lemma}
\begin{proof}
It is a direct consequence of Proposition \ref{alphaintprimfilt}\ref{alphaintprimfilt2}.
\end{proof}
\begin{proposition}\label{apriprithe}
Let $\mathfrak{A}$ be a residuated lattice. The following assertions are equivalent:
\begin{enumerate}
\item  [(1) \namedlabel{apriprithe1}{(1)}] For any $\mathscr{P}\subseteq Spec_{\alpha}(\mathfrak{A})$ and any $\alpha$-filter $F$, $F\subseteq \bigcup \mathscr{P}$ implies $F\subseteq P$ for some $P\in \mathscr{P}$;
\item  [(2) \namedlabel{apriprithe2}{(2)}] for any $\mathscr{P}\subseteq Spec_{\alpha}(\mathfrak{A})$ and any prime $\alpha$-filter $Q$, $Q\subseteq \bigcup \mathscr{P}$ implies $Q\subseteq P$ for some $P\in \mathscr{P}$;
\item  [(3) \namedlabel{apriprithe3}{(3)}] any prime $\alpha$-filter is principal;
\item  [(4) \namedlabel{apriprithe4}{(4)}] any $\alpha$-filter is principal.
\end{enumerate}
\end{proposition}
\begin{proof}
\item [\ref{apriprithe1}$\Rightarrow$\ref{apriprithe2}:] It is obvious.
\item [\ref{apriprithe2}$\Rightarrow$\ref{apriprithe3}:]  Let $Q$ be a prime $\alpha$-filter which is not principal. By Lemma \ref{prielepri} for any $a\in Q$ there exists a prime $\alpha$-filter $P_a$  such that $Q\nsubseteq P_a$. We have $Q\subseteq \bigcup\{P_a|a\in Q\}$ and so $Q\subseteq P_a$ for some $a\in Q$; a contradiction.
\item [\ref{apriprithe3}$\Rightarrow$\ref{apriprithe4}:]  It is obvious that $A$ is a principal $\alpha$-filter. Set $\Psi$ be the set of all proper non-principal $\alpha$-filters of $\mathfrak{A}$. Let $\Psi$ be not empty. In a routine way we can show that $\Psi$ satisfies the conditions of Zorn's lemma. Let $M$ be a maximal element of $\Psi$. Let $a\vee b\in M$ and $a,b\notin M$. Thus there exist $x,y\in A$ such that $\alpha(M,a)=\alpha(x)$ and $\alpha(M,b)=\alpha(y)$. Applying Proposition \ref{veealphafilter}, it follows that $M=\alpha(x\vee y)$; a contradiction. Thus $M$ is a prime $\alpha$-filter and so it is principal; a contradiction. Hence, $\Psi=\emptyset$ and it gets the result.
\item [\ref{apriprithe4}$\Rightarrow$\ref{apriprithe1}:]  Let $F$ be an $\alpha$-filter such that $F\subseteq \cup \mathscr{P}$ for some $\mathscr{P}\subseteq Spec_{\alpha}(\mathfrak{A})$. By hypothesis, $F=\alpha(a)$ for some $a\in A$. So there exists $P\in \mathscr{P}$ such that $a\in P$. Hence, $F\subseteq P$ and it proves the implication.
\end{proof}
\begin{lemma}\label{alphaprlem}
If any coannulet of a residuated lattice is principal, then any its prime $\alpha$-filter is a minimal prime filter.
\end{lemma}
\begin{proof}
Let $P$ be a prime $\alpha$-filter of $\mathfrak{A}$. By Proposition \ref{alpfiltpro}\ref{alpfiltpro2} it is obvious that $P$ has no any dense element. Let $x\in P$. By hypothesis $x^{\perp}=\mathscr{F}(y)$ for some $y\in A$. So we have $(x\odot y)^{\perp}=x^{\perp}\cap y^{\perp}=x^{\perp}\cap x^{\perp\perp}=\{1\}$ and so $x\odot y\notin P$. Hence $y\notin P$ and so $x\in D(P)$. By Theorem \ref{mincor} the result holds.
\end{proof}
\begin{proposition}\label{alphaprpro}
Let $\mathfrak{A}$ be a residuated lattice. The following assertions are equivalent:
\begin{enumerate}
\item  [$(1)$ \namedlabel{alphaprpro1}{$(1)$}] Any $\alpha$-filter is principal;
\item  [$(2)$ \namedlabel{alphaprpro2}{$(2)$}] any $\omega$-filter is principal;
\item  [$(3)$ \namedlabel{alphaprpro3}{$(3)$}] any coannulet is principal and any minimal prime filter is non-dense;
\item  [$(4)$ \namedlabel{alphaprpro4}{$(4)$}] any prime $\alpha$-filter is principal.
\end{enumerate}
Further, any of the above assertions implies that $\mathfrak{A}$ is quasicomplemented.
\end{proposition}
\begin{proof}
\item  [\ref{alphaprpro1}$\Rightarrow$\ref{alphaprpro2}:] It is obvious, by Proposition \ref{alphafofilter}\ref{alphafofilter2}.
\item  [\ref{alphaprpro2}$\Rightarrow$\ref{alphaprpro3}:] By $\gamma(\mathfrak{A})\subseteq \Omega(\mathfrak{A})$ follows that any coannulet is principal. Let $\mathfrak{m}$ be a minimal prime filter. So $\mathfrak{m}$ is an $\omega$-filter and it means that $\mathfrak{m}=\mathscr{F}(x)$ for some $x\in A$. If $\mathfrak{m}$ is dense, follows that $x^{\perp}=1$ and it contradicts with Proposition \ref{profilnoden}.
\item  [\ref{alphaprpro3}$\Rightarrow$\ref{alphaprpro4}:] Let $P$ be a prime $\alpha$-filter. By Lemma \ref{alphaprlem}, $P$ is a minimal prime filter and so it is non-dense. By proposition \ref{nondepri}, $P$ is a coannulet and this means that $P$ is principal.
\item  [\ref{alphaprpro4}$\Rightarrow$\ref{alphaprpro1}:] It  follows by Proposition \ref{apriprithe}.

The rest is evident by Proposition \ref{coannprinstar}.
\end{proof}


\end{document}